\documentclass[journal]{IEEEtran}
\usepackage{amssymb}
\usepackage{graphicx,color}
\usepackage{subfigure}
\usepackage{algorithm}
\usepackage{amsmath,amsthm,graphicx}
\usepackage{cleveref}
\usepackage{epstopdf}
\usepackage{hyperref}

\newtheorem{theorem}{{Theorem}}

\ifCLASSINFOpdf
\else
\fi
\begin{document}
%
\title{A fast two-stage algorithm for non-negative matrix factorization in streaming data}
%
%
%

\author{Ran~Gu,~\IEEEmembership{}
        Qiang~Du,~\IEEEmembership{}
        and~Simon~J.~L.~Billinge~\IEEEmembership{}
\thanks{Ran Gu is with the School of Statistics and Data Science, Nankai University, Tianjin 300071, China.}
\thanks{Qiang Du, and Simon J. L. Billinge are with Department of Applied Physics and Applied Mathematics, Fu Foundation School of Engineering and Applied Sciences, Columbia University, New York, NY 10027, USA.
}
\thanks{Qiang Du is also with the Data Science Institute, Columbia University, New York, NY 10027, USA.
}
\thanks{Simon J. L. Billinge is also with the Condensed Matter Physics and Materials Science Department, Brookhaven National
Laboratory, Upton, NY 11973, USA.}
\thanks{This research was supported by NSF-DMR 1922234 (for R.G., Q.D. and S.B.) and NSF CCF 1740833 (for Q.D.).}
\thanks{Manuscript received January 2th, 2021; revised }}

%
%

\markboth{IEEE TRANSACTIONS ON SIGNAL PROCESSING}%
{Gu \MakeLowercase{\textit{et al.}}: a fast two-stage algorithm for non-negative matrix factorization in streaming data}
%



\maketitle

\begin{abstract}
In this article, we study algorithms for nonnegative matrix factorization (NMF) in various applications involving streaming data.
Utilizing the continual nature of the data, we develop a fast two-stage algorithm for highly efficient and accurate NMF. In the first stage, an alternating non-negative least squares (ANLS) framework is used, in combination with active set method with warm-start strategy for the solution of subproblems. In the second stage, an interior point method  is adopted to accelerate the local convergence. The convergence of the proposed algorithm is proved. The new algorithm is compared with some existing algorithms in benchmark tests using both real-world data and synthetic data. The results demonstrate the advantage of our algorithm in finding high-precision solutions.
\end{abstract}


\ifCLASSOPTIONpeerreview
\begin{center} \bfseries EDICS Category: OPT-NCVX \end{center}
\fi
%
\IEEEpeerreviewmaketitle

\section{Introduction}
%
%
%
%
\IEEEPARstart{N}{onnegative} matrix factorization (NMF) \cite{paatero1994positive,lee1999learning} {refers to the factorization of} a matrix approximately into the product of two nonnegative matrices with low rank, $M\approx XY$. It has become one of the most popular multi-dimensional data processing tools in various applications such as signal processing\cite{buciu2008non}, biomedical engineering\cite{sra2006nonnegative}, pattern recognition\cite{cichocki2009nonnegative}, computer vision and image engineering\cite{buciu2008non}.

Lee and Sewing initiated the study of NMF and presented a method in 1999 \cite{lee1999learning}. Their method
makes all decomposed components non-negative and achieves non-linear dimension reduction at the same time. Developed in \cite{lee2001algorithms} for NMF, Lee and Seung's multiplicative update rule has been a popular method due to the simplicity in implementation.

A commonly used optimization formulation of $M\approx XY$ is to use the Square of Euclidian Distance (SED) as the objective function, that is,
\begin{equation}\label{prob}
\begin{array}{cl}
\min\limits_{X\in \mathbb{R}^{n\times k},Y\in \mathbb{R}^{k\times m}} & \frac{1}{2} \|XY-M\|_F^2\\
                 s.t. & X\geq 0\\
                 & Y\geq 0\\
\end{array}
\end{equation}
Many studies of NMF based on the above formulation have focused  on the use of different optimization approaches like the alternating non-negative least squares (ANLS) \cite{lin2007projected,kim2008nonnegative,guan2012nenmf,huang2015quadratic}, coordinate descent methods \cite{cichocki2009fast,li2009fastnmf}, and alternating direction method of multipliers (ADMM) \cite{hajinezhad2016nonnegative}. A comprehensive survey on various NMF models and many existing NMF algorithms can be found in \cite{wang2012nonnegative}.

The NMF problem has been shown to be non-convex and NP-hard \cite{vavasis2009complexity}.
The algorithms studied in the literature can only guarantee finding a local minimum in general, rather than a global minimum of the cost function.
Although Arora et al. \cite{arora2016computing} presented a polynomial-time algorithm for constant $k$, the order of the complexity of the algorithm is too high to be applied in practice.  Nevertheless, in many data mining applications, high quality local minima are desired in short time \cite{wang2012nonnegative}.

In this study, we are motivated by the application of NMF to problems involving streaming data. Here, streaming data is continually generated data with continuous distributions, such as those from the real-time monitoring of reaction product in chemistry experiments and materials synthesis \cite{liu;jac21, todd;ic20}. For example, Zhao et al. \cite{zhao2011determining} measured the X-ray diffraction data during the nucleation and growth of zeolite-supported Ag nanoparticles through reduction of Ag-exchanged mordenite (MOR), and the processed the data with pair distribution function (PDF) measurement \cite{egami;b;utbp12}. In the data, each PDF is approximately represented by
a vector representation of $n$ dimensions, which is recorded at $m$ time instances in total. We should note the key features of
continual nature of these types of continuously distributed data: at any fixed time, the PDF is continuous in the distance variable; meanwhile, at any fixed distance in the PDF measurement, its value also has continuity in time; moreover, the spatially distributed data at later times are generated progressively following the data earlier in time. In the AgMOR data used in  \cite{zhao2011determining}, the dimensions were respectively $n=3000$ and $m=36$, 
where $n$ is the length of each data, and $m$ is the number of measurements. 
It was anticipated that there are three materials present in the reaction, which means that $k=3$. The focus of our study here is on streaming data in the particular regime of $n\gg m\gg k$, with $k$ being very small. This reflects the high dimensionality of data in an individual measurement (very large $n$) for systems with a relatively small number of  components $k$.

Our goal is to obtain high-precision local solutions for streaming data with a relatively small scale, that $n\leq 5000$, $m\leq 200$, and $k\leq 10$.
Based on our observation, ANLS has a fast descent rate on objective function at the first few iterations, but
a low convergent rate near a local minimizer. Interior point method has a better local convergent rate, but require much more computation. Thus, we propose to combine ANLS with interior point method into a two-stage algorithm. In the first stage,  an active set method is used for ANLS. Due to the continuity of streaming data, the number of active set changes is small. In the second stage,  a line search interior point method is adopted to reach a fast local convergence. Based on the property of streaming data, we optimize the computation of the interior point algorithm so that the computational cost can scale like $O(nm^2k^3)$ in each iteration. The proposed two-stage algorithm combines the advantages of ANLS and interior point method for streaming data.

The organization of this paper is as follows. Section 2 introduces the first stage of our proposed algorithm which is  based on the ANLS with an active set method for streaming data. Section 3 proposes the second stage of the algorithm which is a line search interior point method. Its convergence is also discussed. Section 4 gives the whole framework of our proposed two-stage algorithm to solve NMF in streaming data. Section 5 shows the efficiency of our proposed algorithm by numerical tests. Section 6 concludes this paper with discussions on some future concerns.

%

\section{ANLS framework and active set method}\label{sec:2}

We first briefly review the ANLS framework for solving \eqref{prob}. In ANLS, variables are simply divided into two groups, which are then updated respectively as outlined in \cref{anls}.

\begin{algorithm}
\caption{Alternating nonnegative least squares (ANLS)}\label{anls}
Repeat until stopping criteria are satisfied\\
$$\begin{array}{cl}
\min\limits_{X\in \mathbb R^{n\times k}} & \frac{1}{2} \|XY-M\|_F^2\\
s.t. & X\geq 0.
\end{array}
$$
$$\begin{array}{cl}
\min\limits_{Y\in \mathbb R^{k\times m}} & \frac{1}{2} \|XY-M\|_F^2\\
s.t. & Y\geq 0.
\end{array}
$$
end
\end{algorithm}

Note that each subproblem can be split into a series of nonnegative 
 least square problems
\begin{equation}\label{qp}
\begin{array}{cl}
\min\limits_{x} & \frac{1}{2}\|Cx-d\|_2^2\\
                 s.t. & x\geq 0\\
\end{array}
\end{equation}
For $C=X$, $d$ corresponds to every column of $M$, while for $C=Y^T$, $d$ takes every row of $M$.

Although the original problem in \eqref{prob} is nonconvex, the subproblems in \cref{anls} are convex quadratic problems whose optimal solutions can be found in polynomial time. In addition, the convergence of \cref{anls} to a stationary point of \eqref{prob} has been proved \cite{grippo2000convergence}.

On the basis of ANLS framework, many algorithms for NMF have been developed, such as active set method \cite{kim2008nonnegative}, projection gradient method \cite{lin2007projected}, projection Newton method \cite{gong2012efficient}, projection quasi-Newton method \cite{kim2007fast,zdunek2006non}, Nesterov's gradient method \cite{guan2012nenmf}, and the method combined with Barzilai-Borwein stepsize \cite{huang2015quadratic}.

The active set method solves the subproblem exactly. Kim and park \cite{kim2008nonnegative} introduced an algorithm based on ANLS and active set method. The constrained least square problem in the matrix form with multiple right-hand side vectors is decomposed into several independent nonnegative least square problems with single right-hand side vector \eqref{qp}, which are solved by the active set method of Lawson and Hanson \cite{lawson1995solving}. Later, Kim and Park \cite{kim2011fast} proposed an active-set-like algorithm based on ANLS framework. In their algorithm, the single least squares are solved by block principal pivoting method, and the columns that have a common active set are grouped together to avoid redundant computation of Cholesky factorization in solving linear equations.

By noticing the continuity in streaming data,  when a single least square problem \eqref{qp} is solved, the same $C$ and similar $d$ may be used in the
subsequent least square, resulting in similar solutions with likely the same active set.

Therefore, we choose the active set method of Lawson and Hanson \cite{lawson1995solving} to solve the least square problem \eqref{qp}, using the solution and its active set as the initial guess of the subsequent least square.  This strategy is called warm start strategy or continuation technology, which is widely used in many algorithms and can reasonably improve the effectiveness\cite{liu2016sample,goldfarb2011convergence}. We do not need to do any re-grouping, because the right-hand side vectors have been naturally grouped due to the continuity in data, and we just need to solve them one by one sequentially. In fact, the active set of the solution in the subsequent least square usually has no or little change. When we solve the first set of equations in the new least square subproblem, the Cholesky factorization performed in the previous step can be utilized to avoid redundant computation. Numerically, we see that the active set usually changes only on a very few occasions or does not change at all for the streaming data under consideration. Consequently, the first stage of our algorithm is chosen to be a combination of ANLS framework, active set method and warm start strategy.

\section{A line search interior point method for NMF}

Interior-point methods have proved to be equally successful for various nonlinear optimization as for linear programming. In this section, we propose a line search interior point method for solving NMF. Its global convergence and computational cost are analyzed.

\subsection{Algorithm}

Given that the linear independence constraint qualification (LICQ) holds for NMF problem,  the KKT conditions \cite{kuhn1951} for the problem can be written as
\begin{equation}\label{eqn:kkt}
\begin{array}{c}
(XY-M)Y^T=R\\
X^T(XY-M)=S\\
\langle R,X\rangle=0\\
\langle S,Y\rangle=0\\
X\geq 0,Y\geq 0,R\geq 0,S\geq 0
\end{array}
\end{equation}
We denote $x=vec(X^T)$, $y=vec(Y)$, $r=vec(R^T)$, $s=vec(S)$, where $vec$ transforms a matrix to a vector by expanding it by columns. Meanwhile, we define the inverse operations $mat(x)=X^T$, $mat(y)=Y$, $mat(r)=R^T$, $mat(s)=S$.

Applying Newton's method to the nonlinear system, in the variables $x$, $y$, $r$, $s$, we obtain
\begin{equation}\label{eqn:predictor}
\begin{aligned}
&\left(
  \begin{array}{cccc}
    Q_1 & C^T & -I &  \\
    C & Q_2 &  & -I \\
    Diag(r) &  & Diag(x) &  \\
     & Diag(s) &  & Diag(y) \\
  \end{array}
\right)\left(
         \begin{array}{c}
           \Delta x \\
           \Delta y \\
           \Delta r \\
           \Delta s \\
         \end{array}
       \right)\\
&~~~~~~~~~~~~~~~~~~~~=\left(
                 \begin{array}{c}
                   r-graX \\
                   s-graY \\
                   \mu e-Diag(r)x \\
                   \mu e-Diag(s)y \\
                 \end{array}
               \right)
\end{aligned}
\end{equation}
with $\mu=0$, where $graX=vec(Y(XY-M)^T)$, $graY=vec(X^T(XY-M))$, $e$ is a vector of ones, $Diag(x)$ constructs a diagonal matrix with the diagonal elements given by $x$,
$$Q_1=\left(
       \begin{array}{ccc}
         YY^T &  &  \\
          & \ddots & \\
         &  & YY^T \\
       \end{array}
     \right)
$$
$$Q_2=\left(
       \begin{array}{ccc}
         X^TX &  &  \\
          & \ddots & \\
         &  & X^TX \\
       \end{array}
     \right)$$
\begin{equation}
\begin{aligned}
&C=\left(
       \begin{array}{ccc}
         X_1^TY_1^T & \cdots & X_n^TY_1^T \\
          \vdots& \ddots & \vdots\\
         X_1^TY_m^T& \cdots & X_n^TY_m^T \\
       \end{array}
     \right)\\
&+\left(
       \begin{array}{ccc}
         (X_1Y_1-M_{11})I & \cdots & (X_nY_1-M_{n1})I \\
          \vdots& \ddots & \vdots\\
         (X_1Y_m-M_{1m})I& \cdots & (X_nY_m-M_{nm})I \\
       \end{array}
     \right)
\end{aligned}
\end{equation}
$X_i$ is the $i$th row of $X$, and $Y_j$ is the $j$th column of $Y$.

Let $\mu$ be strictly positive, then the variables $x$, $y$, $r$ and $s$ are forced to take positive values. The trajectory $(x(\mu), y(\mu), r(\mu), s(\mu))$ is called the primal-dual central path. The variables are updated by
\begin{equation}\label{eqn:update}
\begin{array}{c}
x=x+\alpha_1 \Delta x \\
y=y+\alpha_1 \Delta y \\
r=r+\alpha_2 \Delta r \\
s=s+\alpha_2 \Delta s
\end{array}
\end{equation}
where $\alpha_1\in (0,\alpha_1^{\max}]$, $\alpha_2\in (0,\alpha_2^{\max}]$, and
\begin{equation}\label{eqn:alpha}
\begin{array}{rl}\alpha_1^{\max}=&\max \{\alpha\in (0,1]: x+\alpha \Delta x\geq (1-\tau)x,\\
&y+\alpha \Delta y\geq (1-\tau)y\}\\
\alpha_2^{\max}=&\max \{\alpha\in (0,1]: r+\alpha \Delta r\geq (1-\tau)r,\\
&s+\alpha \Delta s\geq (1-\tau)s\}\end{array}\end{equation}
with $\tau\in (0, 1)$. The condition \eqref{eqn:alpha} is called the fraction to the boundary rule \cite{nocedal2006numerical}, which is used to prevent the variables from approaching their lower bounds of 0 too quickly. In this work, we choose $\tau=0.9$.

A predictor or probing strategy \cite{nocedal2006numerical} can also be used to determine the parameter $\mu$. We calculate a predictor (affine scaling) direction
$$(\Delta x^{aff},\Delta y^{aff},\Delta r^{aff},\Delta s^{aff}) $$
by setting $\mu=0$. We probe this direction by letting $(\alpha_1^{aff},\alpha_2^{aff})$ be the
longest step lengths that can be taken along the affine scaling direction before violating
the nonnegativity conditions $(x,y,r,s)\geq 0$. Explicit formulas for these step lengths are given
by \eqref{eqn:alpha} with $\tau= 1$. We then define $\mu^{aff}$ to be the value of complementarity along the
(shortened) affine scaling step, that is,
\begin{equation}\label{eqn:muaff}\begin{array}{rl}
\mu^{aff}=&[(x+\alpha_1^{aff}\Delta x)^T(r+\alpha_2^{aff}\Delta r)\\
&+(y+\alpha_1^{aff}\Delta y)^T(s+\alpha_2^{aff}\Delta s)]/(nk+mk)
\end{array}\end{equation}
and a heuristic choice of $\mu$ is defined as following:
$\mu=\sigma \mu$ and
\begin{equation}\label{eqn:mu}
\sigma=\min\{(\frac{\mu^{aff}}{(x^Tr+y^Ts)/(nk+mk)})^3,0.99\}.
\end{equation}

We propose a two-level nested loop algorithm for the interior point search. In the inner loop, the parameter $\mu$ is fixed. In the outer loop, we gradually reduce $\mu$ to 0.
We use the following error function to break the inner loop, which is based on the perturbed KKT system:
\begin{equation}\label{error}\begin{array}{rl}
E(x,y,r,s;\mu)=&\max\{ \|((graX-r)^T,(graY-s)^T)^T\|,\\
&\|((r^TDiag(x),s^TDiag(Y))-\mu e^T)^T\| \}
\end{array}
\end{equation}


To guarantee the global convergence of the algorithm, we apply a line search approach. First we consider the second derivation of the interior-point method associated with the barrier problem
$$
\min\limits_{X,Y} \frac{1}{2} \|XY-M\|_F^2-\mu \sum \log (X_{ik})-\mu \sum \log (Y_{kj})
$$
We use the exact merit function as same as the barrier function, which can be formed by
$$\begin{array}{rl}
\phi(x,y)=&\frac{1}{2} \|mat(x)^Tmat(y)-M\|_F^2\\
&-\mu \sum \log (x_{i})-\mu \sum \log (y_{j})
\end{array}$$
In the algorithm, we utilize Armijo line search to make the merit function sufficiently decrease.

Considering the convexity of the Hessian, we approximate $$\left(
                                                           \begin{array}{cc}
                                                             Q_1 & C^T \\
                                                             C & Q_2 \\
                                                           \end{array}
                                                         \right) $$
by a positive definite matrix to guarantee the direction $(\Delta x,\Delta y)$ is always a descent direction of $\phi(x,y)$, so that such a line search can be implemented. Notice that the original NMF problem is a least square problem. Thus we can utilize the Hessian in the traditional Guass-Newton algorithm. The Hessian matrix is
$$\left(
                                                           \begin{array}{cc}
                                                             Q_1 & \bar C^T \\
                                                             \bar C & Q_2 \\
                                                           \end{array}
                                                         \right) $$
and it is guaranteed to be positive semi-definite. Here
$$\bar C=\left(
       \begin{array}{ccc}
         X_1^TY_1^T & \cdots & X_n^TY_1^T \\
          \vdots& \ddots & \vdots\\
         X_1^TY_m^T& \cdots & X_n^TY_m^T \\
       \end{array}
     \right)$$
which is the first item of $C$. For further safeguarding, we add a diagonal matrix $\rho I$ to this Hessian, where $\rho$ is a small positive constant. Then we obtain the primal-dual direction by solving
     \begin{equation}\label{eqn:equation}
     \begin{aligned}
&\left(
  \begin{array}{cccc}
    Q_1+\rho I & \bar C^T & -I &  \\
    \bar C & Q_2+\rho I &  & -I \\
    Diag(r) &  & Diag(x) &  \\
     & Diag(s) &  & Diag(y) \\
  \end{array}
\right)\left(
         \begin{array}{c}
           \Delta x \\
           \Delta y \\
           \Delta r \\
           \Delta s \\
         \end{array}
       \right)\\
&~~~~~~~~~~~~~~~~~~~~=\left(
                 \begin{array}{c}
                   r-graX \\
                   s-graY \\
                   \mu e-Diag(r)x \\
                   \mu e-Diag(s)y \\
                 \end{array}
               \right).
\end{aligned}
\end{equation}

By eliminating $\Delta r$ and $\Delta s$ in \eqref{eqn:equation}, we have
\begin{equation}\label{eqn:p}
\left(
  \begin{array}{cc}
    R_1 & \bar C^T \\
    \bar C & R_2 \\
  \end{array}
\right)\left(
  \begin{array}{c}
    \Delta x \\
    \Delta y \\
  \end{array}
\right)=-\nabla \phi(x,y),
\end{equation}
where
$$\begin{array}{c}
R_1=Q_1+\rho I+Diag(x)^{-1}Diag(r),\\
R_2=Q_2+\rho I+Diag(y)^{-1}Diag(s),\\
\nabla \phi(x,y)=(graX^T-\mu(x^{-1})^T,graY^T-\mu(y^{-1})^T)^T,
\end{array}$$
and $Diag(x)^{-1}e$ is simply represented by $x^{-1}$.
We simplify the above formula by
$$
Bp=-\nabla \phi(x,y),
$$
where $B$ represents the coefficient matrix and $p$ represents the direction $(\Delta x^T,\Delta y^T)^T$. Since $B$ is positive definite, the inner product $p^T\nabla \phi(x,y)>0$, which means that $p$ is a descent direction.

To sum up all the approaches above, we present the whole algorithm of interior point method in \eqref{alg:int}. It contains two loops. The parameter $\mu$ is fixed in the inner loop. Due to the Armijo line search \eqref{eqn:armijo} in the inner loop, the stopping criterion of the inner loop can be satisfied in finite iterations. Further, the parameter $\mu$ and $\epsilon_\mu$ are reduced gradually, and by the definition of error function $E$, the solutions of the two-loop algorithm satisfy the KKT system \eqref{eqn:kkt} within the error $\epsilon_{TOL}$. In practice, the barrier stop tolerance can be defined as
$$\epsilon_\mu=\mu.$$
 The complete convergence theorem and its proof are given in \eqref{thm:converge}.

\begin{algorithm}
\caption{A line search interior point method}\label{alg:int}
Initialize: Choose $x_0,y_0,r_0,s_0>0$,  Select an initial barrier parameter $\mu > 0$, parameters $\eta, \sigma\in(0, 1)$, and decreasing tolerances $\epsilon_{\mu}\downarrow 0$ and $\epsilon_{TOL}$. Set $k= 0$.\\
Repeat until $E(x_k,y_k,r_k,s_k;0) \leq \epsilon_{TOL}$\\
\text{\quad} Repeat until $E(x_k,y_k,r_k,s_k;\mu) \leq \epsilon_{\mu}$\\
\text{\quad\quad} Compute the primal-dual direction by solving \eqref{eqn:equation}.\\
\text{\quad\quad} Compute $\alpha_1^{\max}$ and $\alpha_2^{\max}$ using \eqref{eqn:alpha}.\\
\text{\quad\quad} Backtrack step lengths $\alpha_1=\frac{1}{2^t}\alpha_1^{\max}, \alpha_2=\alpha_2^{\max}$\\
\text{\quad\quad} to find the smallest integer $t\geq 0$ satisfying
\begin{equation}\label{eqn:armijo}\begin{array}{l}
\phi (x_k + \alpha_1 \Delta x ,y_k + \alpha_1 \Delta y)\\
 \leq \phi (x_k ,y_k ) + \eta \alpha_1 (\Delta x^T,\Delta y^T)\nabla\phi(x_k,y_k)
 \end{array}
\end{equation}
\text{\quad\quad} Compute $(x^{k+1},y^{k+1},r^{k+1},s^{k+1})$ using \eqref{eqn:update}.\\
\text{\quad\quad} Set $k:=k+1$.\\
\text{\quad} end\\
\text{\quad} Compute parameter $\sigma$ using \eqref{eqn:muaff} and \eqref{eqn:mu} and update $\mu=\sigma \mu$.\\
end
\end{algorithm}

\begin{theorem}\label{thm:converge}
Suppose that all the sequences $\{x_k\}$, $\{y_k\}$, $\{r_k\}$, $\{s_k\}$ generated by \cref{alg:int} are bounded. Then \cref{alg:int} stops in finite iterations.
\end{theorem}
\begin{proof}
We will consider the inner loop first and show that for a given $\mu>0$, $E(x_k,y_k,r_k,s_k;\mu)\leq \epsilon_{\mu}$ will be satisfied in finite iterations.

Based on the Armijo line search rule \eqref{eqn:armijo}, the value of $\phi(x_k,y_k)$ decreases monotonously. Then we have that the lower bound of $x_k$ and $y_k$ is greater than a strictly positive constant that depends on $\mu$. Thus the smallest eigenvalue of the coefficient matrix of \eqref{eqn:p} is greater than a constant greater than 0. Further more, using the boundedness of $(x_k,y_k)$, we obtain that $(\Delta x,\Delta y)$ is bounded. According to the lower bound of $x_k,y_k$, the boundedness of $(\Delta x,\Delta y)$ and \eqref{eqn:alpha}, we have
\begin{equation}\label{prof:amax}\inf \alpha_1^{\max}> 0.
\end{equation}

According to the stepsize rule \eqref{eqn:armijo}, we have
$$
\alpha_1 (\Delta x^T,\Delta y^T)\nabla\phi(x_k,y_k)\rightarrow 0
$$
We aim to prove
\begin{equation}\label{prof:arrow}
(\Delta x^T,\Delta y^T)\nabla\phi(x_k,y_k)\rightarrow 0
\end{equation}
next. To prove it by contradiction, we suppose that
\begin{equation}\label{prof:narrow}
(\Delta x^T,\Delta y^T)\nabla\phi(x_k,y_k) \nrightarrow 0
\end{equation}
This means that there exists a subsequence $\mathcal T$ and a constant $a>0$ such that
\begin{equation}\label{prof:a}
-(\Delta x^T,\Delta y^T)\nabla\phi(x_k,y_k)>a
\end{equation}
for $k\in \mathcal T$. Due to the boundedness of $\{(x_k,y_k)\}_{k\in\mathcal T}$, there exists a subsequence
$\mathcal T_1 \in \mathcal T$ such that $\{(x_k,y_k)\}_{k\in \mathcal T_1}$ converges to $(\bar x,\bar y)$. By \eqref{prof:narrow} and \eqref{prof:a}, we have
$$
\{\alpha_1^k\}_{k\in\mathcal T_1}\rightarrow 0
$$
According to \eqref{prof:amax},
$$\alpha_1^k<\inf \alpha_1^{\max}$$
when $k\in\mathcal T_1$ is large enough. For simplicity, we redefine the sequence satisfying the above condition as $\mathcal T_1$.
From the stepsize rule, the condition \eqref{eqn:armijo} is violated by $\alpha_1 = 2\alpha_1^k$. We have
\begin{equation}\label{prof:frac}\begin{array}{c}
(\phi (x_k + 2\alpha_1^k \Delta x ,y_k + 2\alpha_1^k \Delta y)- \phi (x_k ,y_k ))/(2\alpha_1^k)\\
 > \eta (\Delta x^T,\Delta y^T)\nabla\phi(x_k,y_k)
\end{array}\end{equation}
Taking the limit of the above inequality, we obtain
$$
(\Delta \bar x^T,\Delta \bar y^T)\nabla\phi(\bar x_k,\bar y_k)\geq \eta (\Delta \bar x^T,\Delta \bar y^T)\nabla\phi(\bar x_k,\bar y_k)
$$
Due to $0<\eta<1$, it follows that $(\Delta \bar x^T,\Delta \bar y^T)\nabla\phi(\bar x_k,\bar y_k)\geq 0$. On the other hand, $(\Delta  x^T,\Delta y^T)\nabla\phi( x_k, y_k)<0$. Therefore, $(\Delta \bar x^T,\Delta \bar y^T)\nabla\phi(\bar x_k,\bar y_k)= 0$, which is in contradiction to \eqref{prof:narrow}. So \eqref{prof:arrow} is established.

From \eqref{eqn:p} and \eqref{prof:arrow}, it follows that
$$
\nabla\phi(x_k,y_k)\rightarrow 0
$$
and
$$
(\Delta x,\Delta y)\rightarrow 0.
$$
Then according to \eqref{eqn:equation}, we obtain
$$
\begin{array}{c}
\Delta r\rightarrow \mu x_k^{-1} -r_k\\
\Delta s\rightarrow \mu y_k^{-1} -s_k
\end{array}
$$
For an arbitrary cluster point $(\bar x,\bar y)$ , for any $\delta_\mu >0$, there exists a constant $\kappa$ such that
$$\|(x^{\kappa},y^{\kappa})-(\bar x,\bar y)\|\leq \delta_\mu$$
and
$$
\begin{array}{c}
\|\Delta r+ r_k-\mu x_k^{-1}\|\leq \delta_\mu \\
\|\Delta s+ s_k-\mu y_k^{-1}\|\leq \delta_\mu \\
\|(\Delta x,\Delta y)\|\leq \delta_\mu\\
\|\nabla\phi(x_k,y_k)\|\leq \delta_\mu
\end{array}
$$
for all $k\geq \kappa$.
Due to the boundedness of $(x_k,y_k)$, $\alpha_2^k$ can reach 1 for some $k:=\bar k\leq \kappa+T-1$, where $T$ is a constant. Then it follows from
$$
\|(\bar x,\bar y)-(x_{\bar k},y_{\bar k})\|\leq T\delta_\mu
$$
that
$$
\|r_{\bar k}-\mu x_{\bar k}^{-1}\|\leq c \delta_\mu
$$
$$
\|s_{\bar k}-\mu y_{\bar k}^{-1}\|\leq c \delta_\mu
$$
where $c$ is constant. Therefore, for a given $\epsilon_{\mu}>0$, let $\delta_\mu$ be sufficiently small, then there exists a $k$ such that $E(x_k,y_k,r_k,s_k;\mu)\leq \epsilon_{\mu}$.

We denote the sequence satisfying the inner loop stopping criterion by
$$\{(x_k,y_k,r_k,s_k\}_{k\in \mathcal S}.$$
To prove the theorem by contradiction, we suppose that there is no point satisfying the outer loop stopping criterion $E(x_k,y_k,r_k,s_k;0)\leq \epsilon_{TOL}$. Due to the boundedness, there exists a cluster point of $\{(x_k,y_k,r_k,s_k\}_{k\in \mathcal S}$. For any cluster point $(\bar x,\bar y,\bar r,\bar s)$, we consider the limits on both sides of
 $$E(x_k,y_k,r_k,s_k;\mu)_{k\in \mathcal T_2}\leq \epsilon_\mu,$$
 then we have that
 $$
 E(\bar x,\bar y,\bar r,\bar s;0)=0.
 $$
Thus \cref{alg:int} stops at some ${k\in \mathcal T_2}$, a contradiction. Therefore, \cref{alg:int} stops in finite iterations.
\end{proof}

\subsection{Computation}

In general, the computational cost in each iteration of the interior point method is usually the cubic power of its size, which is impractical for large scale problems. However, for the special NMF problem under consideration here, the amount of computation can be greatly reduced, so that the proposed method can be applied to practical problems involving streaming data. In the following, we analyze the computational cost of \cref{alg:int}.

First of all, to compute the gradient in \eqref{prob}, $O(nmk)$ flops are needed.

The main computational cost of \eqref{alg:int} is computing the primal dual direction \eqref{eqn:equation}. One can solve \eqref{eqn:p} to obtain $(\Delta x,\Delta y)$ first, and then compute $(\Delta r,\Delta s)$ within a low cost.

We rewrite \eqref{eqn:p} as
\begin{equation}\label{eqn:step}
\left(
  \begin{array}{cc}
    \bar Q_1 & \bar C^T \\
    \bar C & \bar Q_2 \\
  \end{array}
\right)
\left(
  \begin{array}{c}
    \Delta x \\
    \Delta y \\
  \end{array}
\right)=\left(
          \begin{array}{c}
            b_1 \\
            b_2 \\
          \end{array}
        \right)
\end{equation}
In order to minimize the computational cost, we first decompose $\bar Q_1$.
$$
\bar Q_1=P^TP
$$
can be obtained by Cholesky factorization or eigenvalue decomposition. Since the matrix $\bar Q_1$ is composed of $n$ positive definite diagonal blocks with the size of k times k, one can obtain $P$ and $P^{-1}$ within $O(nk^3)$ flops.
\eqref{eqn:step} is equivalent to
$$
\left(
  \begin{array}{cc}
    I & P^{-T}\bar C^T \\
    \bar C P^{-1} & \bar Q_2 \\
  \end{array}
\right)
\left(
  \begin{array}{c}
    P\Delta x \\
    \Delta y \\
  \end{array}
\right)=\left(
          \begin{array}{c}
            P^{-T}b_1 \\
            b_2 \\
          \end{array}
        \right).
$$
Then we solve $\Delta y$ from
\begin{equation}\label{eqn:step2}
(\bar Q_2-(\bar C P^{-1})(\bar C P^{-1})^T)\Delta y=b_2-(\bar C P^{-1})(P^{-T}b_1)
\end{equation}
and compute $\Delta x$ by
$$
\Delta x=P^{-1}(P^{-T}b_1-(\bar C P^{-1})^T\Delta y).
$$

We need $O(nmk^2)$ flops for constructing $\bar C$ and $O(nmk^3)$ flops for computing $\bar CP^{-1}$. By considering that each block of $\bar C$ is rank one, we can compute $\bar CP^{-1}$ within $O(nmk^2)$ flops. The dominant computation is the computation of $(\bar C P^{-1})(\bar C P^{-1})^T$ which costs $O(nm^2k^3)$ flops. If we consider that each block of $\bar C$ is rank one, it can be reduced to $O(nm^2k^2)$ flops.
When solving \eqref{eqn:step2} by Cholesky factorization, since the size of the coefficient matrix is $mk$ by $mk$, the computational cost is $O(m^3k^3)$.
Other computations, like computing the right side of \eqref{eqn:step2} and computing $\Delta x$, are $O(nmk^2)$ flops.

To sum up, the computation cost of \cref{alg:int} in each iteration is $O((n+mk)m^2k^2)$. Compared with the computation of gradient $O(nmk)$, the cost is no more than $O(mk^2)$ times. Since $n\gg m\gg k$ and $k$ is usually very small in our streaming data, the computational complexity is completely acceptable.

\section{A two-stage algorithm for NMF on streaming data}

In this section, we propose a practical algorithm with fast convergence for solving NMF in streaming data. It combines both ANLS framework with active set method and interior point method proposed in the previous sections.

In the early stage of the algorithm, we use ANLS framework with active set method. It can reduce the value of objective function rapidly. We use the relative step tolerance, which is a relative lower bound on the size of a step, meaning
$$\|(x_k,y_k)-(x_{k+1},y_{k+1})\|\leq \epsilon_{STOL} (1 + \|(x_k,y_k)\|),$$
as the stopping criterion of this stage. If the algorithm attempts to take a step that is smaller than step tolerance, the iterations end.

At the end of the first phase, the algorithm is going to enter the second phase of using the interior point method. However, the solutions from active set method usually contain elements with value zero, which is incompatible to the strict interior point required by interior point method. Meanwhile, we also need to provide the initial dual variables $(r,s)$ to the interior point method. To address these issues, we first change the primal variable smaller than $\rho_0$ to $\rho_0$ by
\begin{equation}\label{upxy}
(x,y):=\max\{(x,y),\rho_0\}
\end{equation}
where $\rho_0$ is a small positive constant. We can choose
\begin{equation}\label{rho0}
\rho_0=10^{-6}\max\{(x,y)\}.
\end{equation}
Next, we give the initial value of the dual variable by
\begin{equation}\label{rs}
\begin{array}{c}
r=\max\{|graX|\}e\\
s=\max\{|graY|\}e
\end{array}
\end{equation}
We set the parameters
\begin{equation}\label{mu}
\mu=\frac{x^Tr+y^Ts}{mk+nk}
\end{equation}
and
\begin{equation}\label{rho}
\rho=\epsilon_{TOL}.
\end{equation}

After these preparations, the algorithm enters its second stage by implementing \cref{alg:int}. As the Hessian matrix is approximated in \cref{alg:int} by the positive definite matrix
\begin{equation}\label{hessian2}
\left(
  \begin{array}{cc}
    Q_1+\rho I & \bar C \\
    \bar C & Q_2+\rho I \\
  \end{array}
\right),
\end{equation}
in order to further speed up convergence, we can change $\bar C$ back to $C$ at the right time. A heuristic way to switch to
\begin{equation}\label{hessian}
\left(
  \begin{array}{cc}
    Q_1+\rho I & C \\
     C & Q_2+\rho I \\
  \end{array}
\right),
\end{equation}
is by monitoring $\sigma$, which is given in \eqref{eqn:mu}. A small $\sigma$ implies that the predictor step generates a point close to the boundary, thus it is likely to be close to a local minimum. Therefore, when
$$
\sigma\leq \sigma_c,
$$
we switch to \eqref{hessian}, where $\sigma_c$ is a user-supplied constant, and we set $\sigma_c=0.01$ in our test. The computation of the primal-dual direction is similar to that using $\bar C$. The difference is that each block of $C$ is no longer rank one, thus the computational cost is $O(nm^2k^3)$ flops. Since \eqref{hessian} is not guaranteed to be positive definite, the primal-dual direction using \eqref{hessian} may not be a descent direction. We check the negativity of
$$
(\Delta x^T,\Delta y^T)\nabla\phi(x_k,y_k)
$$
in \eqref{eqn:armijo} before we implement the line search. If we fail to obtain a descent direction, we switch to the positive definite Hessian \eqref{hessian2}.

All elements of the above approaches are presented in our two-stage algorithm, as outlined below in \cref{alg:fast}.
\begin{algorithm}
\caption{A fast two-stage algorithm}\label{alg:fast}
Initialize: Choose initial $X,Y$\\
Implement \eqref{anls} with active set method\\
Update variables by \eqref{upxy}, \eqref{rho0} and \eqref{rs}\\
Set parameters by \eqref{mu} and \eqref{rho}, $\eta=0.5$, select the tolerance $\epsilon_{TOL}$, and let $flag=0$.\\
Repeat until $E(x_k,y_k,r_k,s_k;0) \leq \epsilon_{TOL}$\\
\text{\quad} $\epsilon_\mu:=\mu$\\
\text{\quad} Repeat until $E(x_k,y_k,r_k,s_k;\mu) \leq \epsilon_{\mu}$\\
\text{\quad\quad} If $flag=1$, \\
\text{\quad\quad\quad} compute the primal-dual direction by solving \eqref{eqn:equation}
\text{\quad\quad\quad} with Hessian \eqref{hessian}.\\
\text{\quad\quad\quad} If $(\Delta x^T,\Delta y^T)\nabla\phi(x_k,y_k)\geq 0$,\\
\text{\quad\quad\quad\quad} $flag=0$, compute the primal-dual direction by\\
\text{\quad\quad\quad\quad} solving \eqref{eqn:equation}.\\
\text{\quad\quad} Else, \\
\text{\quad\quad\quad} compute the primal-dual direction by solving \eqref{eqn:equation}.\\
\text{\quad\quad} Compute $\alpha_1^{\max}$ and $\alpha_2^{\max}$ using \eqref{eqn:alpha}.\\
\text{\quad\quad} Backtrack step lengths $\alpha_1=\frac{1}{2^t}\alpha_1^{\max}, \alpha_2=\alpha_2^{\max}$ to find\\
\text{\quad\quad} the smallest integer $t\geq 0$ satisfying \eqref{eqn:armijo}.\\
\text{\quad\quad} Compute $(x^{k+1},y^{k+1},r^{k+1},s^{k+1})$ using \eqref{eqn:update}.\\
\text{\quad\quad} Set $k:=k+1$.\\
\text{\quad} End\\
\text{\quad} Compute parameter $\sigma$ using \eqref{eqn:muaff} and \eqref{eqn:mu} and update $\mu=\sigma \mu$.\\
\text{\quad} If $\sigma\leq 0.01$, $flag=1$, else $flag=0$\\
End
\end{algorithm}

\section{Numerical tests}
\label{sec:5}

We test our two-stage algorithm and compare it with other ANLS-based methods including NeNMF \cite{guan2012nenmf}, QRPBB \cite{huang2015quadratic} and ANLS-BPP \cite{kim2011fast}. All the tests are performed using MATLAB 2020a on a computer with 64-bit system and 2.70GHz CPU. Comparisons are done on both synthetic data sets and real-world problems.

\subsection{Stopping criterion}

The KKT conditions of \eqref{prob} are given in \eqref{eqn:kkt}. The definition of the error function $E(x,y,r,s;0)$ \eqref{error} measures the violation of the KKT conditions. Therefore, we set
$$E(x,y,r,s;0)\leq \epsilon_{TOL}$$
to be the stopping criterion, where
$$\epsilon_{TOL}=10^{-6}.$$

The ANLS-based algorithms do not generate the dual variables $r$ and $s$. Here we give a reasonable definition that
$$
\begin{array}{c}
r=\max\{graX,0\}\\
s=\max\{graY,0\}\\
\end{array}
$$

In some cases, we also limit the maximum CPU time, in cases where some algorithms can not reach the given accuracy.

\subsection{Streaming AgMOR PDF data}

Zhao et al. \cite{zhao2011determining} measured the X-ray diffraction data during the nucleation and growth of zeolite-supported Ag nanoparticles through reduction of Ag-exchanged mordenite (MOR), and processed the data with pair distribution function (PDF) measurement. In the field of chemistry, more and more people use mathematical tools to analyze their measured data. Chapman et al. \cite{chapman2015applications} used principle component analysis (PCA) and some post-processing to analyze the data given in  \cite{zhao2011determining}, and obtained 3 principle components. Since PDF is a distribution-type function, NMF may be intuitively more applicable.

We simply remove the negativity of the raw data by shifting each PDF up by the opposite of its original minimum value. Then we perform the NMF algorithms on the data. The size of the data is
$$n=3000,\ m=36,$$
and in the algorithms, $k$ is set to be 3 based on the analysis of \cite{chapman2015applications}.

We use the same initial point for each algorithm. The relationship between KKT violation $E$ and CPU time of the four algorithms are shown in Fig.~\ref{fig:agmor}. It can be seen that the performances of NeNMF, ANLS-BPP and our 2-STAGE algorithm are similar in the early stage, and the decline of error function value for QRPBB is the most obvious. After that, there is a rise of error function value in 2-STAGE, because our algorithm begins to enter the second stage and the variables change, and some shocks occur later, which is caused by the instability of the early iterations of the interior point method. After several iterations, 2-STAGE becomes stable and converges quickly to meet the termination criterion. Compared with NeNMF and ANLS-BPP, QRPBB is always more accurate. Even so, it can not achieve the given criterion in a short time.

\begin{figure}
  \centering
  \includegraphics[scale=0.5]{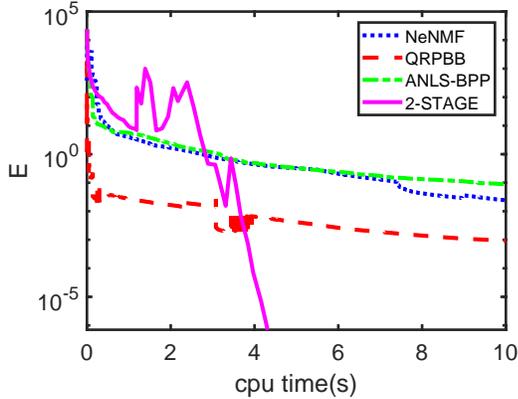}\\
  \caption{CPU time (s) versus tolerance values on AgMOR Data}\label{fig:agmor}
\end{figure}

Next, we generate 10 different random initial points. For each of them, all algorithms are implemented. The results are shown in Tab.~\ref{tab:agmor}. Because none of the three algorithms compared with the 2-STAGE algorithm can meet the stopping criterion in a short time, we set the maximum running CPU time as 20 seconds. Tab.~\ref{tab:agmor} presents the average, minimum and maximum CPU time, the KKT violation $E$ and the objective function value of each algorithm respectively. It can be seen that our algorithm always has the minimum KKT violation and the minimum objective function value. It has a great advantage in finding a high-precision solution within a short amount of time.

\begin{table*}[htbp]
		\centering
\setlength{\tabcolsep}{2mm}
{
\begin{tabular}{|c|c|c|c|}
  \hline
  Algorithm & cpu(s):avrg(min,max) & E:avrg(min,max) & f:avrg(min,max) \\
  \hline
  NeNMF & 20.01(20.00,20.03) & 6.64(1.72,16.32)e-2& 5.37274(5.37242,5.37329)e+2 \\
  QRPBB &20.02(20.00,20.03) & 2.19(1.48,2.94)e-3& 5.37236(5.37213,5.37263)e+2 \\
  ANLS-BPP & 20.02(20.02,20.05) & 1.37(0.75,2.14)e-2 & 5.37236(5.37206,5.37264)e+2 \\
  2-STAGE & 5.03(3.55,8.88) & 3.58(1.24,6.30)e-7 & 5.37205(5.37205,5.37205)e+2 \\
  \hline
\end{tabular}}
\caption{Experimental results on AgMOR data}\label{tab:agmor}
\end{table*}

\subsection{Yale face database}

The Yale face database is widely tested for face recognition \cite{cai2006orthogonal}. It contains 165 grayscale images of 15 individuals. There are 11 images per subject, one per different facial expression or configuration: lighting (center-light, left-light and right-light), with/without glasses, facial expressions (normal, happy, sad, sleepy, surprised, and wink). The original image size is $320\times 243$ pixels. To reduce the computational burden, the image size was reduced to $64\times 64$ pixels. Because the image of face also has some degree of continuity, we regard the data as streaming data to test our algorithm. The data size is $n=4096,m=165$.

We gradually increase $m$ to explore the influence of $m$ variation on the algorithm. We selected the face images of the first $i$ individuals, $i=1,2,4,8,15$, i.e. $m=11,22,44,88,165$. In the following tests, we fix $k=3$. We limit a maximum CPU time of 60 seconds for each algorithm and generate 10 different random initial points for each sample. We see from Fig.~\ref{fig:yaleface}, where we take an example of $m = 44$ and an example of $m = 165$, that the first stage of 2-STAGE is not as efficient as other algorithms. The main reason is that the continuity of face data is not strong enough compared with streaming data, thus the active set changes frequently in the algorithm resulting large computational cost. In Fig.~\ref{fig:yaleface} (top), the second stage of 2-STAGE, interior point method, shows a fast local convergence, since the slope is significantly steeper than the others. However, in Fig.~\ref{fig:yaleface} (bottom), the slope of the second stage of 2-STAGE is similar with others. The reason is that when $m$ increases, the computational cost of interior point method increases faster than that of other algorithms. Therefore, we see from Tab.~\ref{tab:yaleface} that, as anticipated, the larger $m$ is, the less efficient 2-STAGE is compared with other algorithms.

\begin{figure}\centering
\includegraphics[scale=0.5]{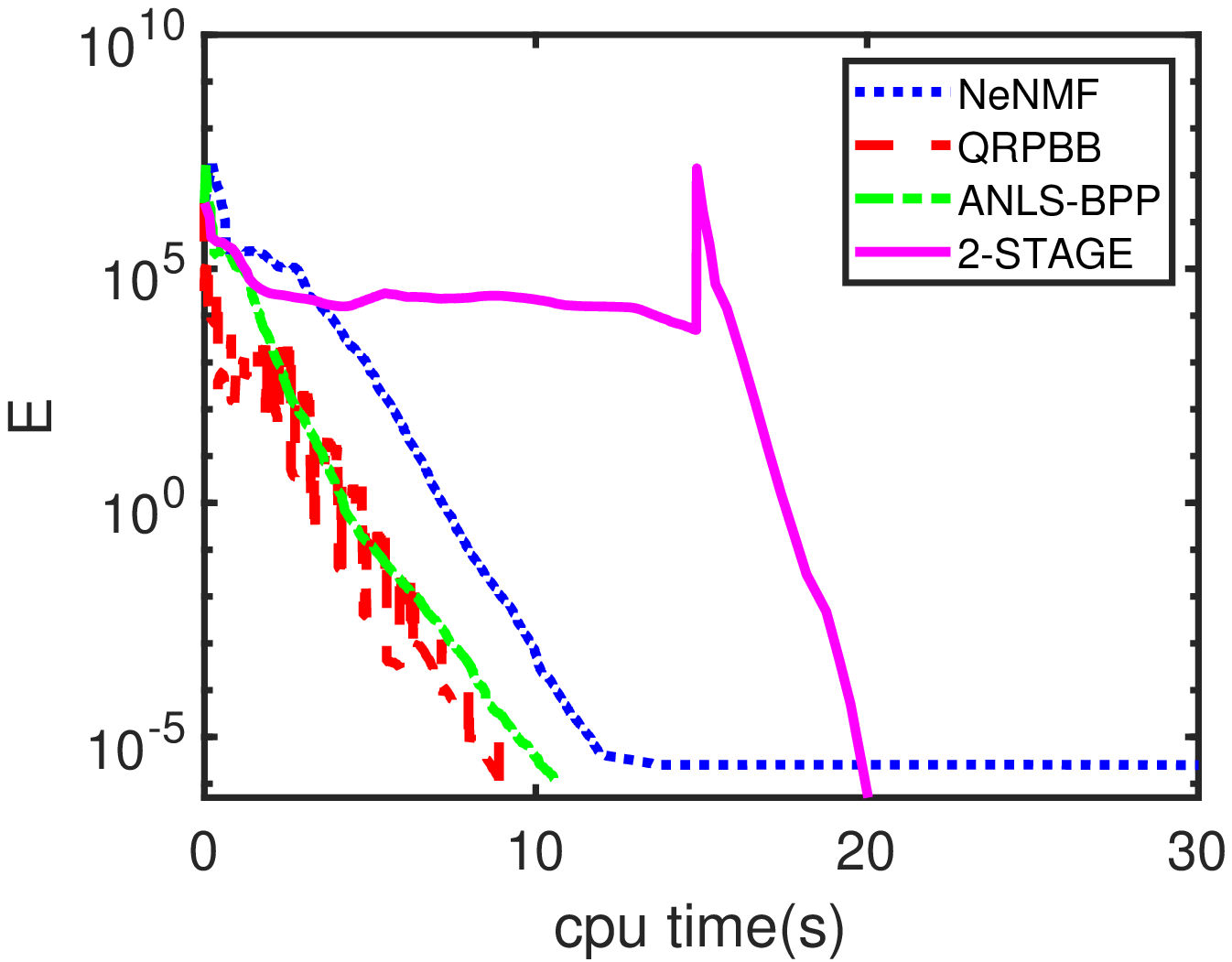}\\
\includegraphics[scale=0.5]{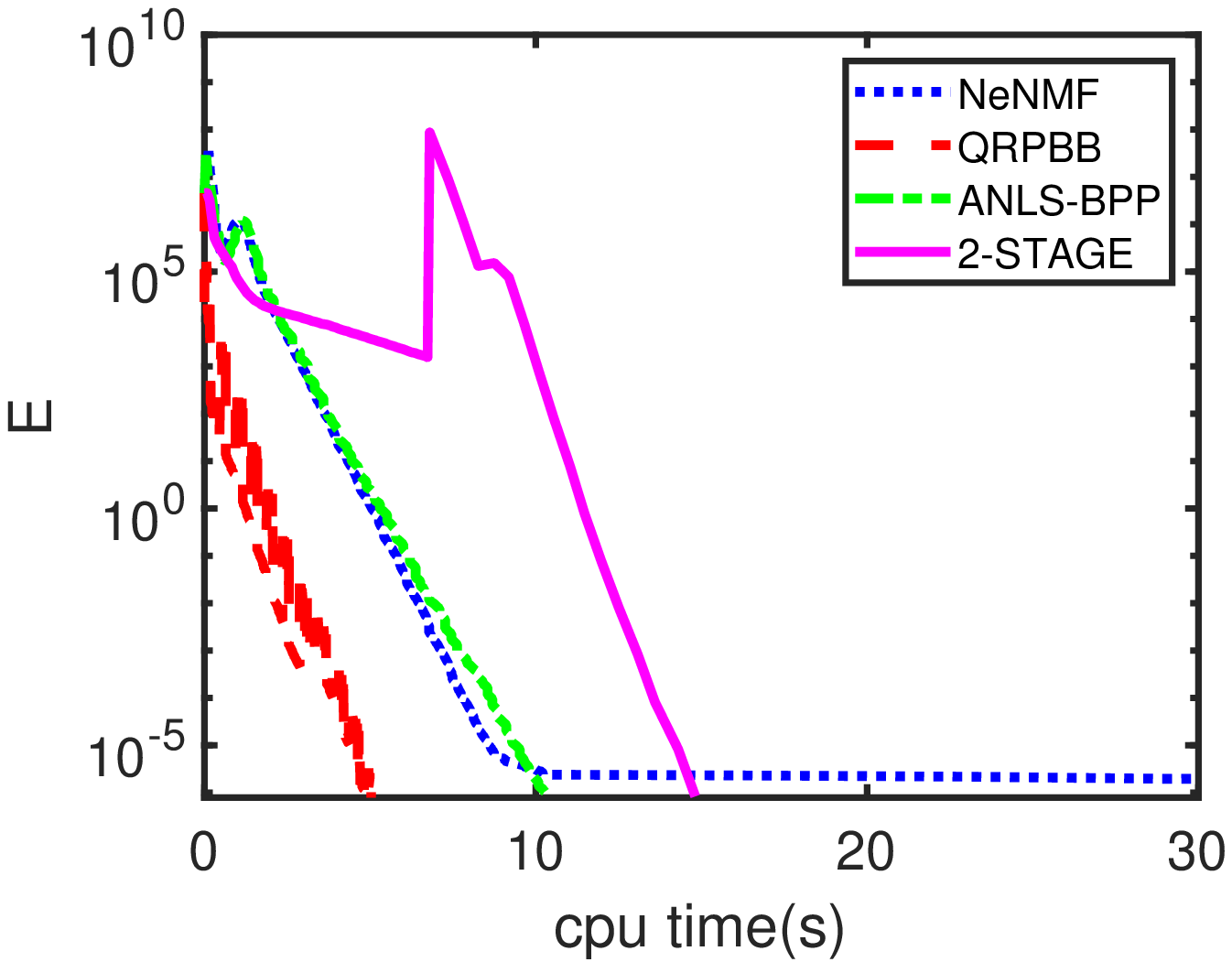}
\caption{Experimental tests on Yale face data, $m=44$ (top) and $m=165$ (bottom).}
\label{fig:yaleface}
\end{figure}

\begin{table*}[htbp]
		\centering
\setlength{\tabcolsep}{2mm}
{
\begin{tabular}{|c|c|c|c|c|}
  \hline
  size(n,m,k)&Algorithm & cpu(s):avrg(min,max) & E:avrg(min,max) & f:avrg(min,max) \\
  \hline
  (4096,11,3)&NeNMF & 44.94(30.50,66.53) & 125(9.77e-7,1254)& 1.40186(1.40104,1.40918)e+7 \\
  &QRPBB &19.10(12.05,40.28) & 8.20(3.39,9.99)e-7& 1.40511(1.40104,1.40918)e+7 \\
  &ANLS-BPP & 23.14(16.42,49.08) & 9.80(9.50,9.98)e-7 & 1.40185(1.40104,1.40918)e+7 \\
  &2-STAGE & 13.30(6.83,24.33) & 4.39(1.16,9.46)e-7 & 1.40520(1.40104,1.41008)e+7 \\
  \hline
  (4096,22,3)&NeNMF & 36.36(7.28,65.63) & 1.13(9.47e-7,11.34)& 3.13573(3.12627,3.14786)e+7 \\
  &QRPBB &6.30(2.41,17.23) & 7.32(2.63,9.95)e-7& 3.13324(3.12627,3.14785)e+7 \\
  &ANLS-BPP & 19.90(2.89,60.03) & 0.79(9.48e-7,5.78) & 3.13772(3.12627,3.14786)e+7 \\
  &2-STAGE & 9.81(5.45,18.28) & 3.54(1.79,7.37)e-7 & 3.12893(3.12627,3.14785)e+7 \\
  \hline
  (4096,44,3)&NeNMF & 63.76(61.03,67.06) & 2.71(2.04,6.17)e-6& 8.20694(8.20694,8.20694)e+7 \\
  &QRPBB &6.32(4.97,7.06) & 8.52(2.49,9.99)e-7& 8.20694(8.20694,8.20694)e+7 \\
  &ANLS-BPP & 10.06(8.34,12.14) & 9.85(9.44,9.99)e-7 & 8.20694(8.20694,8.20694)e+7 \\
  &2-STAGE & 16.05(11.08,31.14) & 5.05(1.02,9.02)e-7 & 8.20694(8.20694,8.20694)e+7 \\
  \hline
  (4096,88,3)&NeNMF & 62.86(61.05,64.81) & 2.09(1.35,4.19)e-6& 1.97873(1.97753,1.98962)e+8 \\
  &QRPBB &4.97(3.14,10.58) & 8.37(5.42,9.96)e-7& 1.97945(1.97753,1.98962)e+7 \\
  &ANLS-BPP & 9.10(7.25,16.50) & 4.53(1.32,9.81)e-7 & 1.97874(1.97753,1.98962)e+7 \\
  &2-STAGE & 14.76(7.86,43.48) & 5.05(1.02,9.02)e-7 & 1.97753(1.97753,1.97753)e+7 \\
  \hline
  (4096,165,3)&NeNMF & 63.99(60.17,67.20) & 1.54(1.35,1.89)e-6& 4.06305(4.06305,4.06305)e+8 \\
  &QRPBB &5.91(5.30,6.78) & 5.01(4.79,9.58)e-7& 4.06305(4.06305,4.06305)e+8 \\
  &ANLS-BPP & 11.46(10.81,12.36) & 9.60(9.40,9.85)e-7 & 4.06305(4.06305,4.06305)e+8 \\
  &2-STAGE & 29.78(16.58,61.78) & 4.10(1.19,8.39)e-7 & 4.06305(4.06305,4.06305)e+8 \\
  \hline
\end{tabular}}
\caption{Experimental results on Yale face database}\label{tab:yaleface}
\end{table*}

\subsection{Synthetic data}

In order to test more problems on different scales, we artificially synthesized some data for testing.

We note that the focus of this paper is on streaming data. Chapman et al.\cite{chapman2015applications} use PCA method to study AgMor data. They found that there were three dominant components (singular values) of the data, which means that the original data is a rank-3 matrix plus noise. According to this characteristics, we construct the artificial data by the following methods. First, we decide the problem size $(n,m,k)$. Then, we generate a random matrix $X$ whose size is $n\times k$, and each element is uniformly distributed from 0 to 1. In the same way, we generate a random matrix $Y$ whose size is $k\times m$. Next, we compute $M=XY$ and add Gaussian noise, whose expectation is 0 and standard deviation is $0.1$, to each element of $M$. Finally, we change the negative elements in $M$ to zeros.

In this set of tests, each algorithm is limited to a maximum CPU time of 60 seconds. We generate 10 different random initial points for sample. The results are shown in Tab.~\ref{tab:synthetic}. From an average point of view, the performance of 2-STAGE is the best of the four algorithms, followed by QRPBB. Except for M = 50, k = 6, ANLS-BPP exceeds QRPBB. In some cases, other algorithms may be better than 2-STAGE. For example, when $m = 100$ and $k = 6$, the minimum time of QRPBB is much less than the maximum time of 2-STAGE. Generally speaking, the performance of 2Stage in getting high-precision solution truly stands out.

\begin{table*}[htbp]
		\centering
\setlength{\tabcolsep}{2mm}
{
\begin{tabular}{|c|c|c|c|c|}
  \hline
  size(n,m,k)&Algorithm & cpu(s):avrg(min,max) & E:avrg(min,max) & f:avrg(min,max) \\
  \hline
  (2000,50,3)&NeNMF & 60.01(60.00,60.03) & 1.74(0.70,3.57)e-5& 4.61662(4.61662,4.61662)e+2 \\
  &QRPBB &29.36(20.17,48.14) & 8.46(5.60,9.80)e-7& 4.61662(4.61662,4.61662)e+2 \\
  &ANLS-BPP & 58.48(47.89,60.05) & 1.69(1.00,3.29)e-6 & 4.61662(4.61662,4.61662)e+2 \\
  &2-STAGE & 3.31(2.45,4.75) & 3.07(1.16,5.86)e-7 & 4.61662(4.61662,4.61662)e+2 \\
  \hline
  (2000,50,4)&NeNMF & 47.80(40.88,60.02) & 1.19(1.00,2.31)e-6& 4.56019(4.56019,4.56019)e+2 \\
  &QRPBB &22.89(12.78,34.22) & 7.30(2.67,9.99)e-7& 4.56019(4.56019,4.56019)e+2 \\
  &ANLS-BPP & 28.03(23.44,32.61) & 9.99(9.98,10.00)e-7 & 4.56019(4.56019,4.56019)e+2 \\
  &2-STAGE & 7.41(4.22,11.78) & 3.25(1.09,8.91)e-7 & 4.56019(4.56019,4.56019)e+2 \\
  \hline
  (2000,50,5)&NeNMF & 49.81(40.92,56.80) & 9.99(9.98,10.00)e-7& 4.45449(4.45449,4.45449)e+2 \\
  &QRPBB &25.14(20.61,35.91) & 5.99(1.80,8.86)e-7& 4.45449(4.45449,4.45449)e+2 \\
  &ANLS-BPP & 24.78(21.67,28.31) & 9.99(9.97,9.99)e-7 & 4.45449(4.45449,4.45449)e+2 \\
  &2-STAGE & 6.98(6.20,7.91) & 3.93(1.69,8.31)e-7 & 4.45449(4.45449,4.45449)e+2 \\
  \hline
  (2000,50,6)&NeNMF & 60.02(60.00,60.03) & 5.18(0.38,13.11)e-3& 4.36121(4.36121,4.36121)e+2 \\
  &QRPBB &60.01(60.00,60.03) & 2.19(0.02,8.31)e-4& 4.36121(4.36121,4.36121)e+2 \\
  &ANLS-BPP & 56.07(20.38,60.05) & 5.35(0.10,32.81)e-5 & 4.36309(4.36121,4.38005)e+2 \\
  &2-STAGE & 14.46(6.31,25.81) & 4.74(1.18,7.77)e-7 & 4.36121(4.36121,4.36121)e+2 \\
  \hline
  (2000,100,3)&NeNMF & 17.14(14.88,19.91) & 9.96(9.93,10.00)e-7& 9.57337(9.57337,9.57337)e+2 \\
  &QRPBB &5.84(4.86,7.70) & 7.81(1.76,9.98)e-7& 9.57337(9.57337,9.57337)e+2 \\
  &ANLS-BPP & 15.10(14.27,15.97) & 9.98(9.94,10.00)e-7 & 9.57337(9.57337,9.57337)e+2 \\
  &2-STAGE & 4.05(2.95,5.33) & 4.71(1.00,9.92)e-7 & 9.57337(9.57337,9.57337)e+2 \\
  \hline
  (2000,100,4)&NeNMF & 26.86(23.59,31.17) & 9.98(9.97,10.00)e-7& 9.54230(9.54230,9.54230)e+2 \\
  &QRPBB &14.53(13.17,15.63) & 5.75(2.37,8.99)e-7& 9.54230(9.54230,9.54230)e+2 \\
  &ANLS-BPP & 24.60(22.36,26.70) & 9.98(9.96,10.00)e-7 & 9.54230(9.54230,9.54230)e+2 \\
  &2-STAGE & 7.46(5.36,8.91) & 4.82(1.01,9.44)e-7 & 9.54230(9.54230,9.54230)e+2 \\
  \hline
  (2000,100,5)&NeNMF & 36.04(31.03,41.45) & 9.98(9.97,10.00)e-7& 9.44857(9.44857,9.44857)e+2 \\
  &QRPBB &17.60(16.17,20.45) & 7.02(4.22,9.12)e-7& 9.44857(9.44857,9.44857)e+2 \\
  &ANLS-BPP & 27.54(25.70,29.58) & 9.98(9.96,10.00)e-7 & 9.44857(9.44857,9.44857)e+2 \\
  &2-STAGE & 11.77(9.69,17.59) & 4.59(1.27,8.06)e-7 & 9.44857(9.44857,9.44857)e+2 \\
  \hline
  (2000,100,6)&NeNMF & 60.02(60.00,60.03) & 7.43(0.03,28.52)e-3& 9.35451(9.35451,9.35452)e+2 \\
  &QRPBB &57.20(37.47,60.03) & 2.54(0.07,12.20)e-5& 9.35451(9.35451,9.35451)e+2 \\
  &ANLS-BPP & 60.03(60.02,60.05) & 1.14(0.04,4.44)e-4 & 9.35451(9.35451,9.35451)e+2 \\
  &2-STAGE & 25.54(14.44,52.39) & 3.05(1.04,9.33)e-7 & 9.35451(9.35451,9.35451)e+2 \\
  \hline
\end{tabular}}
\caption{Experimental results on synthetic data}\label{tab:synthetic}
\end{table*}

\section{Conclusions}
\label{sec:6}

In this paper we focused on solutions to the NMF for streaming data. We presented a fast two-stage algorithm, where the first stage is the ANLS framework with active set method which gains benefit from the continuity of streaming data, and the second stage is a line search interior point method which gets benefit from $n\gg m\gg k$. In addition, we have proved the global convergence of the proposed line search interior point method. The first stage reduces the value of the objective function rapidly, and the second stage converges to a local solution quickly due to the property of Newton-type direction. We tested the proposed algorithm on several real and synthetic data, and observed that compared with other algorithms, our algorithm is more effective in solving high-precision local solutions.

The active set method in the first stage does not reach the expected speed, even if it is tested on continuous data. We think that this may be caused by the limitations of the underlying code implementation in MATLAB. On the other hand, we find that the transition part between the two stages may induce instability. This is because the solution of the active set method cannot be directly used as the initial guess of the interior point method, and its changes have an impact on stability.
At present, the parameters used to generate starting point are selected carefully to avoid the instability. In the future, we will work to find a more stable transition technique.

Considering that in addition to the basic NMF model, there are other variants of NMF, such as constrained NMFs and structured NMFs, our algorithm has the potential to be applications to more problems through extension. This will be our consideration in the future.

\end{document}